\documentclass[11pt,a4paper]{article}
\usepackage{a4wide,amsfonts,amsmath,latexsym,amssymb,euscript,graphicx,units,mathrsfs}
\numberwithin{equation}{section}

\usepackage{graphicx}
\usepackage{color}
\usepackage{mathtools}
\usepackage{amssymb}
\usepackage{amssymb}
\usepackage[T1]{fontenc}
\usepackage{latexsym}
\usepackage{xypic}
\usepackage{eufrak}
\usepackage{euscript}
\usepackage{verbatim}
\usepackage{fancyhdr}
\usepackage[english]{babel}

\usepackage{units}

\newtheorem{prop}{Proposition}[section]

\newtheorem{lemme}[prop]{Lemma}
\newtheorem{rem}[prop]{Remark}

\newtheorem{theorem}[prop]{Theorem}

\newcommand{\E}{\mathbb{E}}

\renewcommand{\geq}{\geqslant}
\def\leq{\leqslant}

\def\var{{\mathbb{Var}}}

\def\1{{\mathbf{1}}}

\def\1{{\mathbf{1}}}
\def\0.5{{\frac{1}{2}}}
\def\var{{\rm{Var}}}

\begin{document}

 \pagestyle{plain}

\title{\textbf{On the Besov-Orlicz path regularity  of  some Gaussian processes}}   

\author{
  Rachid Belfadli \footnote{Corresponding author. Department of Mathematics, Faculty
of Sciences and Technologies,
Laboratory of Mathematics, Artificial
Intelligence and Sustainable Technologies,Cadi Ayyad University 2390 Marrakesh, Morocco. E-mail: rachid.belfadli@uca.ac.ma  }, \quad  \quad Brahim Boufoussi  \footnote{Department of Mathematics, Faculty of Sciences Semlalia, Cadi Ayyad University, B.P. 2390, Marrakesh, 40000, Morocco. E-mail: boufoussi@uca.ac.ma} \quad and 
  \quad Youssef Ouknine \footnotemark[2]\, \thanks{Mohammed VI Polytechnic University- Rabat Campus Rocade Rabat-Sal\'{e}  11103 Morocco.  E-mail: youssef.ouknine@um6p.ma} \quad   
}
\maketitle

\begin{abstract}  In this paper, we rely on the additive decomposition in law satisfied by a class of stochastic processes, combined with the well-known regulariy properties of fractional Brownian motion, to establish  Besov-Orlicz  regularity of their sample paths. This provides a unified and direct  proof for a broad class of processes, including bifractional Brownian motion  with parameters $H\in (0, 1]$, $ K\in (0, 2)$ such that $HK \in (0, 1)$, subfractional Brownian motion with Hurst parameter $H\in (0, 1)$, and certain  class of self-similar processes. %associated with the stochastic  heat equation.
\end{abstract}
    


\section{Introduction} 
Besov spaces are general framework  to study the modulus of smoothness for paths of continuous stochastic processes.  There is an extensive body of literature devoted to the study of Besov or Besov-Orlicz path regularity for stochastic processes; see, for instance, the non-exhaustive list of references \cite{NB3, NB, CKR, CO, V}

In  \cite{NB},  Boufoussi and Nachit have improved the H\"{o}lder continuity results for the bifractional Brownian motion (bfBm in short) $\{B^{H, K}_t, t \in [0, 1]\}$ of Hurst parameters $H, K \in (0, 1)$ by showing that  almost all their sample paths belong to the Besov-Orlicz space $\mathcal{B}^{HK}_{\Phi_2, \infty}([0, 1])$, where $\Phi_2$ stands  for the Young function $\Phi_2(x):= e^{x^2}-1$ (see below for precise statement and notations). Their proof method relies on intricate  Ciesielski et al's  isomorphism characterization of the Besov spaces $\mathcal{B}^{HK}_{p, \infty}([0 .1])$ in terms  sequences spaces (see,  \cite{CKR}).

The starting point of the present work is the observation that many Gaussian processes of interest admit a decomposition in law into the sum of a fractional Brownian motion and a smooth auxiliary process. This type of decomposition, first established by Lei and Nualart \cite{LN} for the bfBm and second by Ruiz De Ch\'{a}vez and Tudor \cite{RuizT} for the subfractional Brownian motion (sfBm), provides a powerful structure insight; it allows one to transfer many regularity properties from the fBm component to Gaussian processes with more complex covariance structure in a direct and transparent manner.
 
Motivated by this idea, the main objective of this paper is to develop a unified and direct approach to Besov-Orlicz regularity for a broad class of Gaussian processes. Indeed, instead of relying on Ciesielski et al's  isomorphism characterization, we exploit probabilistic decompositions combined with sharp moment estimates. This in particular leads to proofs that are more simpler.  We show that the Besov-Orlicz  path regularity of several processes including bfBm, sfBm and certain class of self-similar Gaussian processes can be deduced  from the well-known one of fBm and the additive decomposition in law of the process under consideration. 
 In particular, for the bfBm we state a more general formulation that holds for all parameters $H \in (0, 1)$  and $K\in (0, 2)$ such that $ HK \in (0, 1)$. From this perspective, the present work can be viewed as a natural complement to  \cite{NB}. 
 
 The paper is organized as follows. In Section $2$, we recall some preliminaries on  Besov and Besov-Orlicz spaces.  We then present  a variety of stochastic processes that can be still split  in law into the sum of a fBm and Lei-Nualart's process; namely, the bfBm, sfBm and a class of self-similar Gaussian processes. Section $3$ is devoted to the statement of the main results and their proofs.
\section{Preliminaries} \label{sec2}
Throughout this paper we assume that all the processes considered are  defined in a complete probability space $(\Omega, \mathcal{F}, \mathbb{P})$.

Let $I=(0, 1)$ and for $h\in \mathbb{R}$, we denote $I(h):=\{s\in I: s+h \in I\}$. For  $\alpha \in (0, 1)$ and $p, q \in [0, +\infty)$, we will recall briefly in this subsection the definition of Besov, Exponential Orlicz and Besov-Orlicz spaces that are with our concerns, with the usual modification for $p=+\infty$ or $q=+\infty$. A detailed treatment of this topic can be found in Triebel  \cite{Tr}. 
\subsection{Besov space}\label{BesovSpace} 
The Besov spaces  $\mathcal{B}^{\alpha}_{p, q}(I)$, with $\alpha \in (0, 1)$, are general framework  to study the modulus of smoothness in $L^p$-norms for paths of continuous stochastic processes. More precisly, the  Besov space $\mathcal{B}^{\alpha}_{p, q}(I)$ is defined as the space of functions $f\in L^p(I)$ for which the seminorm
\begin{eqnarray}\label{seminorm1}
(f)_{\mathcal{B}^{\alpha}_{p, q}(I)}:= \left( \int_{0}^{1}( t^{-\alpha}\omega_p(f, t))^q \dfrac{dt}{t}\right)^{1/q}
\end{eqnarray}
is finite, where 
\[\omega_p(f, t):=\sup_{|h|\leq t}\left(\int_{I(h)}|f(s+h)-f(s)|^p ds \right)^{1/p},%\;\;\; \; \mbox{for}\,\, 0<t\leq 1.
\]
The Besov spaces $\mathcal{B}^{\alpha}_{p, q}(I)$ are  Banach spaces when endowed with the sum of $L^p$-norm and this seminorm. It can be shown  (see, for instance \cite{CO}) that the above seminorm is equivalent to 
\[ ||f||_{p, q, \alpha}:= \left( \sum_{n \geq 0} 2^{n q \alpha} ||f(\cdot +2^{-n}) - f(\cdot)||^{q}_{L^p(I(2^{-n}))} \right)^{1/q}.
\]
It is then more convenient to take the norm
\begin{eqnarray}\label{norm_1}
||f||_{p, q, \alpha}:=||f||_{L^p(I)} +  ||f||_{p, q, \alpha}
\end{eqnarray}
as a Banach space norm on  $\mathcal{B}^{\alpha}_{p, q}(I)$. The particular case when $q=+\infty$ corresponds to the Besov space $\mathcal{B}^{\alpha}_{p, \infty}(I)$ endowed with the norm
\begin{eqnarray}\label{norm_2} || f||_{p, \infty, \alpha}:= ||f||_{L^{p}(I)} + \sup_{n\geq 0} 2^{n \alpha}||f(\cdot +2^{-n}) - f(\cdot)||_{L^p(I(2^{-n}))}. \end{eqnarray}

\subsection{Exponential Orlicz and Besov-Orlicz spaces}\label{O-B-Spaces} 
Let us consider the Young function $\Phi_{\beta}(x):= e^{x^{\beta}}-1$, for $\beta >0$. The exponential Orlicz space is defined as the space of all measurable functions $f \in L^{0}(I)$ for which the norm
\[ ||f ||_{L^{\Phi_{\beta}}(I)}:= \sup_{p\geq 1} p^{-1/\beta} ||f ||_{L^{p}(I)}
\]
is finite.  The Besov-Orlicz space $\mathcal{B}^{\alpha}_{\Phi_{\beta}, \infty}(I)$  is defined as the space 
\[ \mathcal{B}^{\alpha}_{\Phi_{\beta}, \infty}:= \{ f\in L^{0}(I): (f)_{\alpha, \beta}:= \sup_{0<t\leq 1} (t^{-\alpha} \tilde{\omega}_p(f, t))<\infty\},
\]
where $\tilde{\omega}_p(f, t):=\sup_{|h|\leq t} ||f(\cdot + h)- f(\cdot)||_{L^{\Phi_{\beta}}(I)}$.  The space $\mathcal{B}^{\alpha}_{\Phi_{\beta}, \infty}(I)$  is a Banach space when endowed with the sum of the $L^{\Phi_{\beta}}(I)$-norm and the seminorm $(\cdot)_{\alpha, \beta}$. The seminorm $(\cdot)_{\alpha, \beta}$ is equivalent to 
\[
||| f|||_{\alpha, \beta}:= \sup_{n \geq 1} 2^{n \alpha} ||f(\cdot + 2^{-n}) -f( \cdot)||_{L^{\Phi_{\beta}}(I(2^{-n}))}.
\]
For the purpose below it is more convenient to choose the norm
\begin{eqnarray}\label{norm-OB_2}
||f||_{\mathcal{B}^{\alpha}_{\Phi_{\beta}, \infty}}:=||f ||_{L^{\Phi_{\beta}}(I)} + ||| f|||_{\alpha, \beta}
\end{eqnarray}
as a Banach space norm on $\mathcal{B}^{\alpha}_{\Phi_{\beta}, \infty}(I)$.
\begin{rem}\label{remk} Notice that from the definitions of the above spaces the following embedding holds
\[ \mathcal{B}^{\alpha}_{\Phi_{\beta}, \infty} \subseteq \mathcal{B}^{\alpha}_{p, \infty}, \qquad \mbox{for every}\quad p \in [1, \infty).\]
Furthermore,
\begin{eqnarray}\label{BO-embeding}
\mathcal{B}^{\alpha}_{\Phi_{\beta}, \infty} \subseteq \mathcal{B}^{\gamma}_{\Phi_{\beta}, \infty}, \quad \mbox{if}\quad \gamma \leq \alpha. 
\end{eqnarray}
\end{rem}

\subsection{Examples of suitable processes}\label{processes}
In this subsection, we present several basic examples of processes 
that admit a decomposition in law as the sum of a fBm and the Lei-Nualart process, including the sfBm and a class of self-similar Gaussian processes associated with the stochastic heat equation, introduced by Harnett and Nualart \cite{HN}. Various properties of bBm, sfBm and  numerous other self-similar processes can be found in \cite{T13}
\subsubsection{Bifractional Brownian motion}
The bi-fractional Brownian motion is a generalization of the fractional Brownian motion (bfBm), first introduced by  Houdr\'{e} and Villa \cite{HV}. It is  defined as a centred Gaussian process $\{B^{H, K}_t, t\geq 0\}$, with the covariance function 
\begin{eqnarray}\label{cov-bfBm}
R^{H, K}(t, s)=2^{-K}((t^{2H}-s^{2H})^K-|t-s|^{2HK}),\;\;\;\; s, t \geq 0,
\end{eqnarray}
where $H\in(0, 1)$ and $K\in (0, 1]$. Note that, if $K=1$ then  $B^{H}:=B^{H, 1}$ is simply the fBm with Hurst parameter $H$. We refer the reader to \cite{RT} and \cite{LN} for further properties. 

The following decomposition  in law of the bfBm into the sum of a fBm and a regular stochastic process have been proved by Lei and Nualart \cite{LN}
\begin{eqnarray}\label{LNdecompo}
\{\; c_2 B^{HK}_t, t\geq 0 \; \}\overset{d}{=}\{\; B^{H, K}_t+c_1X_t^{H, K}, t\geq 0 \; \},
\end{eqnarray}
where $c_1=\sqrt{ \frac{K2^{-K}}{\Gamma(1-K)}}$, $c_2= 2^{\frac{1-K}{2}}$, $\{B^{HK}, t\geq 0\}$ is a fBm with Hurst parameter $HK$ and the process $X^{H, K}$ is defined, for $0<K<1$, by
\begin{eqnarray}\label{LNprocess1}
X^{H, K}_0=0\,\,\,\, \mbox{and}\,\,\,\,X^{H, K}_t:=X^{K}_{t^{2H}}
\end{eqnarray}
with
\begin{eqnarray}\label{LNprocess}
X^{K}_t:=\int_{0}^{\infty} (1-e^{-\theta t})\theta^{-\frac{1+K}{2}}dW_{\theta},\,\,\,\,\,\mbox{for all}\,\,\,\, t>0,
\end{eqnarray}
and $\{W_{\theta},  \theta \geq 0\}$ is a standard Brownian motion independent of $B^{HK}$. The notation  $X \overset{d}{=} Y$ means that the two processes $X$ and $Y$ have the same distribution.\\
Using  the process $X^{H, K}$, which is well defined for all $0<K<2$, Bardina and Es-Sebaiy \cite{BE}  showed that $R^{H, K}$ still a covariance function for $0<H<1$ and $1<K<2$ such that $0<HK<1$, and hence extending the definition of the bfBm to that region.  Mainly, they showed that if $W$ and $B^{HK}$ are independents then the process $Y^{H, K}$ defined by
\begin{eqnarray}\label{BE-ext}
Y_t^{H, K}:= aB_t^{HK} +bX_t^{H, K},\;\;\;\; t\geq 0
\end{eqnarray}
is a bfBm of parameters  $H$ and $K$, that is, a centered Gaussian process with the covariance function given by (\ref{cov-bfBm}). Here, the real constants $a$ and $b$ are given by $a=\sqrt{2^{1-K}}$ and $b=\sqrt{\frac{K(K-1)}{2^K\Gamma(2-K)}}$.
\begin{rem}
Notice that the exact range of pairs $(H, K)$ for which the function in (\ref{cov-bfBm}) is a covariance function on $\mathbb{R}_{+}^2$ is actually investigated in  Lifshits and Volkova \cite{LV} and Talarczyk \cite{Tal20} .
\end{rem}

\subsection{Subfractional Brownian motion.}  The sub-fractional Brownian motion (sfBm) with  parameter $H\in(0, 1)$, introduced in Bojdecki et al. \cite{Bo-all},  is defined as the mean-zero  Gaussian process $\{S_t^H, t\geq 0\}$, with covariance function given by
\begin{eqnarray}\label{cov-subfBm}
R_{H}(t, s)=s^{2H}+ t^{2H}-\frac{1}{2}\left[(s+t)^{2H} + |s-t|^{2H} \right], \;\;\;\; s, t \geq 0.
\end{eqnarray}
 Note that, when $H=1/2$, $S^{1/2}$ is reduced to the standard Brownian motion. For further discussion and  properties we refer the reader to \cite{Bo-all, RuizT, T},  among other references. For instance, in \cite{RuizT}  the authors obtained the following decomposition in law of the sfBm depending on the value of $H$: %{\color{red} Put $2k+1=2H$}
\begin{itemize}
\item if $H\in (0, 1/2)$, then with the process $X^{K}$ given by (\ref{LNprocess}), we have
\begin{eqnarray}\label{subdecomp1}
\{S^H_t, \;\; t\geq 0 \; \}\overset{d}{=}\{\; B^{H}_t+c_3X_t^{2H}, t\geq 0 \; \},
\end{eqnarray}
where $B^{H}$ is a fBm with Hurst parameter $H$ independent of the standard Brownian motion $W$ and  $c_3= \sqrt{\frac{H(1-2H)}{\Gamma(2-2h)}}$.
\item if $H\in(1/2, 1)$,  then
\begin{eqnarray}\label{RT-ext} 
\{B^{H}_t, \;\; t\geq 0 \; \}\overset{d}{=}\{\; S^H_t+c_4X_t^{2H}, t\geq 0 \; \},
\end{eqnarray}
where $S^{H}$ is a sfBm independent of $W$ with Hurst parameter $H$ and  $c_4=\sqrt{\frac{H(2H-1)}{\Gamma(2-2H)}}$.
\end{itemize}

\subsection{A class of self-similar processes.} Another  example of processes where the analogous  additive decomposition in law occurs is the one introduced by Harnett and Nualart in \cite{HN}. Consider the centred Gaussian process $G=\{G_t, \; t \geq 0\}$ with covariance 
\begin{eqnarray}\label{HN-covariance}
R(t, s):= \mathbb{E}(G_tG_s)=\mathbb{E}\left(\int_0^t Z_{t-r}dB^{H}_r\right)\left(\int_0^s Z_{s-r}dB^{H}_r\right),
\end{eqnarray} 
where $B^H$ is a fBm with Hurst parameter $H\in(0, 1)$ and $Z=\{ Z_t, \; t>0\}$ is a zero mean Gaussian process, independent of $B^H$, with covariance $\mathbb{E}(Z_tZ_s)= (t+s)^{-\gamma}$, where $\gamma \in(0, 2H)$. Notice  here that the process $\{\int_0^t Z_{t-r} dB^{H}_r,\; t\geq 0\}$ is a non Gaussian process having the same covariance function as the process $\{G_t,\; t\geq 0\}$.  

According to Harnett and Nualart \cite{HN} the following two properties are satisfied:
\begin{itemize}
\item[1)]When $H\in (\frac{1}{2}, 1)$, then $G$ has the law  of the solution in time of a stochastic heat equation. Indeed, let $d\geq 1$ be an integer and consider the $d$-dimensional stochastic heat equation
\[\frac{\partial u}{\partial t}=\frac{1}{2}\Delta u \; + \dot {W}^{H}, \; t\geq 0, \; x\in \mathbb{R}^{d}, \]
with zero initial condition, where $\dot{W}^{H}$ is a fractional-colored noise with covariance given by
\[\mathbb{E}(\dot{W}^H(t, x)\dot{W}^H(s, y))=c_{d, \beta} \alpha_H |t-s|^{2H-2}|x-y|^{-\beta},\]
where $0<\beta<d \wedge 2$, $ \alpha_H=H(2H-1)$ and $c_{d, \beta}=\pi^{-d/2}2^{\beta-d}\Gamma(\beta/2)/\Gamma(\frac{d-\beta}{2})$.

Now, with 
$$D=(2\pi)^{-d}(1-\beta/2)^{-1}\int_{\mathbb{R}^d} \frac{e^{-\frac{|\xi|^2}{2}}}{|\xi|^{d-\beta}}d \xi,$$ 
the covariance  function of the process $\{u(t, 0),\; t\geq 0\}$ is given by
\[\mathbb{E}(u(t, 0)u(s, 0))= D\alpha_H \int_0^t\int_0^s |u-v|^{2H-2}(t+s -u-v)^{-\gamma} dudv,\]
which  is exactly, up to a constant,  the covariance of the process $G$ when $H\in(\frac{1}{2}, 1)$.
\item[2)] The process $G$ can be split in law as the sum of a fBm and a Lei-Nualart type process, namely we have
\begin{eqnarray}\label{HNdecompo}
\{\; G_t, t\geq 0 \; \}\overset{d}{=}\{\; \sqrt{\kappa} B^{\alpha/2}_t+\sqrt{\lambda} X_t^{2\alpha +1}, t\geq 0 \; \},
\end{eqnarray}
where $B^{\alpha/2}$ is a fBm with Hurst parameter $\alpha
/2$, $X^{2\alpha +1}$ is the Lei-Nualart process defined by (\ref{LNprocess}) and the constants $\alpha$, $\kappa$ and $\lambda$ are given by
\[\alpha=2H-\gamma,\;\; \kappa=\frac{1}{\Gamma(\gamma)}\int_{0}^{\infty} \frac{z^{\gamma -1}}{1+z^{2}}dz \;\;\;\mbox{and}\;\lambda=\frac{4\pi}{\Gamma(\gamma)\Gamma(2H+1)\sin(\pi H)}\int_{0}^{\infty}\frac{\eta^{1-2H}}{1+\eta^2}d\eta.\]
%\item[3)]
\end{itemize}

\section{Main results}
Since the processes we consider in this paper belong to the class of processes that can be split in law into two part that are belonging each other  to $\mathcal{B}^{\alpha}_{\Phi_2, \infty}(I)$, our results will be easily  deduced from the following simple observation whose proof is straightforward  and whence omitted.

\begin{prop}\label{prop1}Let $X=\{X_t, \; t \in I\}$, $Y=\{Y_t, \; t \in I\}$ and $Z=\{Z_t, \; t \in I\}$ be three stochastic processes such that $Z\overset{d}{=}X+Y$ with $Y \in  \mathcal{B}^{\alpha}_{\Phi_2, \infty}(I)$ for some $\alpha \in (0, 1)$. Then, almost surely all the  paths of the process
$X$ belong to $ \mathcal{B}^{\alpha}_{\Phi_2, \infty}(I)$  if and only if  
almost surely all the  paths of the process $Z$ belong to $\mathcal{B}^{\alpha}_{\Phi_2, \infty}(I)$.
\end{prop}
\begin{rem}\label{rem2}
All the examples of processes  considered in Section~\ref{sec2} fall within the scope of Proposition~\ref{prop1}. Therefore, it suffices to verify that almost surely the paths of $X^{H, K}$ belong to $\mathcal{B}^{HK}_{\Phi_2, \infty}(I)$.
\end{rem}

Our first main result can be stated as follows.

\begin{theorem}
\label{Lei-NualartProcess-result2} 
Let  $H\in (0, 1)$ and $K\in (0, 2]$ such that $0<HK<1$. Then,  the paths of  $X^{H, K}$ belong almost surely to the Orlicz-Besov space $\mathcal{B}^{HK}_{\Phi_2, \infty}(I)$.
\end{theorem}
\textit{Proof:} We have to show that $||X^{H, K}_{\cdot}||_{\mathcal{B}^{\alpha}_{\Phi_{\beta}, \infty}}< \infty$  a.s. That is \\
\begin{eqnarray} \label{reg:bo}
 \sup_{p\geq 1}\dfrac{1}{\sqrt{p} }||X^{H, K}||_{L^p(I)} <\infty \quad \, \, \mbox{ and}\,\, \quad \sup_{p\geq 1}\dfrac{1}{\sqrt{p} } \sup_{n \in \mathbb{N}^{\star}}Y_{n, p}<\infty, \, \, \mbox{ a.s.,} 
\end{eqnarray}
where $Y^{p}_{n, p}$ is given by
\begin{eqnarray}\label{the process-: Yn,p}
Y^{p}_{n, p}:= 2^{nHKp} || X^{H, K}_{\cdot +2^{-n}}-X^{H, K}_{\cdot} ||^p_{L^p(I(2^{-n}))} =2^{nHKp}\displaystyle \int_{0}^{1-2^{-n}} \left | \int_{t^{2H}}^{(t+2^{-n})^{2H}}Y_u du\right|^{p}dt.
\end{eqnarray}
We start by proving the first statement. First, notice that for fixed $t>0$, we have
\begin{align}\label{moment p:1}
\E(|X_t^{H, K}|^p)=\E\left(\left| \int_{0}^{\infty}\frac{1-e^{-\theta t^{2H}}}{\theta^{\frac{1+K}{2}}} dW_{\theta}\right|^{p}\right) =c_p^p C^{p/2}_Kt^{HKp},
\end{align}
with
$c_p^p:=\E(|W_1|^p)= \dfrac{2^{p/2}\Gamma(\frac{p+1}{2})}{\sqrt{\pi}}$
and%\footnote{Indeed, by using Fubini's theorem we have 
%\begin{eqnarray*}
%C_K&=& \displaystyle\int_{0}^{\infty} \theta^{1-K}\displaystyle\int_{0}^{1}\displaystyle\int_{0}^{1} e^{-(x+y)\theta} dx dy d \theta=\displaystyle\int_{0}^{1}\displaystyle\int_{0}^{1} e^{-(x+y)\theta} dx dy\left(\displaystyle\int_{0}^{\infty} \theta^{1-K} e^{-(x+y)\theta}d \theta\right)\\
%&=& \Gamma(2-K)\int_{0}^{1}\int_{0}^{1} (x+y)^{K-2} dx dy.
%\end{eqnarray*}
%}
\begin{eqnarray*}
C_K&:=&\displaystyle\int_{0}^{\infty}\frac{(1-e^{-\theta})^2}{\theta^{1+K}}d\theta= \displaystyle\int_{0}^{\infty} \theta^{1-K}\displaystyle\int_{0}^{1}\displaystyle\int_{0}^{1} e^{-(x+y)\theta} dx dy d \theta\\
&=&\displaystyle\int_{0}^{1}\displaystyle\int_{0}^{1} e^{-(x+y)\theta} dx dy\left(\displaystyle\int_{0}^{\infty} \theta^{1-K} e^{-(x+y)\theta}d \theta\right)\\
&=&\Gamma(2-K)\int_{0}^{1}\int_{0}^{1} (x+y)^{K-2} dx dy.
\end{eqnarray*}
That is
 \begin{eqnarray}
C_K:=\displaystyle\int_{0}^{\infty}\frac{(1-e^{-\theta})^2}{\theta^{1+K}}d\theta= \Gamma(2-K)\times \begin{dcases}
	2 \log(2),&\quad  \mbox{if}\quad \,  K=1; \\
	\dfrac{2^K-2}{K(K-1)},&   \mbox{if}\quad \,   K\neq 1.
	\end{dcases}  
\end{eqnarray}
Then
\[\dfrac{1}{\sqrt{p} } \int_{0}^{1} \E(|X^{H, K}_t|^p) dt= \dfrac{c_p^p C_K^{p/2}}{\sqrt{p}(1+ HKp)}.\]
Applying Markov's inequality we obtain,  for all $\lambda >0$,
\begin{eqnarray*}\label{}
\mathbb{P}(||X^{H, K} ||_{L^p(I)} > \lambda c_p)&\leq& \dfrac{c_{p}^{-2p}}{\lambda^{2p}}\E\left( ||X^{H, K} ||^{2p}_{L^p(I)}\right)\\
&\leq &\dfrac{c^{-2p}}{\lambda^{2p}}\int_{0}^{1} \E(|X^{H, K}_t|^{2p}) dt \leq  c_p^{-2p}c_{2p}^{2p}\left(\dfrac{\sqrt{C_K}}{\lambda}\right)^{2p},
\end{eqnarray*}
where we have used Cauchy-Schwarz's inequality to get the second inequality.\\ Using Stirling's formula, we infer that $c_{p}^{-2p}c_{2p}^{2p}\leq c$, for some constant $c>1$ and thus  for $ \lambda = 2(c \sqrt{C_K} +1)$, we obtain  $\sum_{p\geq 1}\mathbb{P}(||X^{H, K} ||_{L^p(I)} > \lambda c_p)< +\infty,$ which  in turn implies, by the Borel-Cantelli's lemma, that 
$$\mathbb{P}\left(\dfrac{||X^{H, K} ||_{L^p(I)}}{c_p} > \lambda \, \, \mbox{i.o.}\right)=0.$$
Therefore,  $\sup_{p \in \mathbb{N}^{\star}}\dfrac{1}{\sqrt{p} }||X^{H, K}||_{L^p(I)} <+\infty$ a.s., since $c_p \geq c \sqrt{p}$ for some  constant $c>1$.\\
For real $p\geq 1$, because $[p]+1 \leq 2p$, we have 
\[\dfrac{1}{\sqrt{p} }||X^{H, K}||_{L^p(I)} \leq  \dfrac{\sqrt{2}}{\sqrt{[p] +1 } }||X^{H, K}||_{L^{[p] +1}(I)} \leq \sup_{q\in \mathbb{N}^{\star}}\sqrt{2}\dfrac{1}{\sqrt{q} }||X^{H, K}||_{L^q(I)} < +\infty \quad \mbox{a.s.} \]
Consequently $\sup_{p\geq 1}\dfrac{1}{\sqrt{p} }||X^{H, K}||_{L^p(I)} <+\infty$ a.s. \\
Concerning the second assertion, we let for $s >0$, $0<h<1$ and $0\leq t \leq 1-h$,
\[Y_s:= \int_{0}^{+\infty} e^{-\theta s}\theta^{-\frac{K-1}{2}}dW_{\theta}\quad \mbox{and}\quad F_h(t):= \int_{t^{2H}}^{(t+h)^{2h}}Y_u du.\]
Clearly, we have  for $u, v >0$,
\[\E(Y_uY_v)= \int_{0}^{+\infty} e^{-\theta (u+v)}\theta^{1-K} d \theta= \Gamma(2-K)(u+v)^{K-2}.\]
Then
\begin{eqnarray} \label{eq:moments}
\E(|F_h(t)|^p)= c_p^p \var(F_h(t))^{p/2}=c_p^p\Gamma^{p/2}(2-K)\left(\int_{t^{2H}}^{(t+h)^{2h}} \int_{t^{2H}}^{(t+h)^{2h}} (u+v)^{K-2}du dv \right)^{p/2}.
\end{eqnarray}
We will need the following ancillary result whose proof is postponed until the proof of Theorem \ref{Lei-NualartProcess-result2} is finished.
\begin{lemme}\label{lem:1} There exists two positive constants $A_{H, K}$ and $B_{H, K}$ such that  for all $0<h<1/2$ and $t \in [0, 1-h]$ 
\begin{eqnarray} \label{claim:1}
\E(|F_h(t)|^p)\leq \begin{dcases}
	(A_{H, K} c_p)^p h^{HKp},&\quad  \mbox{if}\quad \,  t\in [0, h]; \\
	(B_{H, K} c_p)^p h^{p}t^{-(1-HK)p},&   \mbox{if}\quad \,  t\in [h, 1-h].
	\end{dcases}  
\end{eqnarray}
\end{lemme}
Now taking $h_n=2^{-n}$ for $n \geq 1$ in Lemma \ref{lem:1}  and integrating (\ref{claim:1}) over $[0, 1-2^{-n}],$ we get 
\begin{eqnarray*} 
\E(Y^p_{n, p}) &=& 2^{npHK }\int_{0}^{1-2^{-n}} \E[ |F_{2^{-n}}(t)|^p] dt\notag\\
&\leq & (A_{H, K} c_p)^p 2^{-n}  + (B_{H, K} c_p)^p 2^{-np\lambda} \int_{2^{-n}}^{1} t^{- \lambda p} dt,
\end{eqnarray*}
with $\lambda:= 1-HK \in (0, 1)$.  Then
\begin{eqnarray*} 
\E(Y_{n, p}^p)&\leq&  \begin{dcases}
	(A_{H, K} c_p)^p 2^{-n} + (B_{H, K} c_p)^p 2^{-np\lambda} n\log(2) & ,  \mbox{if}\quad \,  \lambda p \leq 1;\nonumber \\
	(A_{H, K} c_p)^p 2^{-n} + (B_{H, K} c_p)^p \dfrac{2^{-n} }{\lambda p -1} & ,  \mbox{if}\quad \,  \lambda p > 1 \nonumber.\\
	\end{dcases} 
\end{eqnarray*}
Therefore 
\begin{eqnarray} \label{eq: estimation moments}
\E(Y_{n, p}^p)
\leq  (c_p C_{H, K})^p \eta_{n, p},
\end{eqnarray}
with $C_{H, K}:= A_{H, K} \vee B_{H, K} $ and 
\begin{eqnarray*} 
\eta_{n, p}&:=&  \begin{dcases}
	2^{-n}  +  2^{-np\lambda} n\log(2) & ,  \mbox{if}\quad \,  \lambda p \leq 1;\nonumber \\
	2^{-n} +  \dfrac{2^{-n} }{\lambda p -1} = \dfrac{\lambda p}{\lambda p -1}2^{-n}& ,  \mbox{if}\quad \,  \lambda p>1 \nonumber.\\
	\end{dcases} 
\end{eqnarray*}
Now, let  $p_0=[1/\lambda] + 1$ and observe that, on one hand, if $p\geq 1$ is such that $\lambda p >1$,  then 
\[\eta_{n, p}=\left(1+ \dfrac{1}{\lambda p -1}\right)2^{-n} \leq 2^{-n}\left(1+ \dfrac{1}{\lambda p_0 -1}\right)=:2^{-n} \delta_{H, K}.\]
Implying $\sum_{n}\eta_{n, p} \leq 2 \delta_{H, K}$. \\
\noindent On the other hand, if $ p \geq 1 $  is such that $\lambda p <1$, then  $$\eta_{n, p}=2^{-n} + \log(2)n2^{-\lambda np} \leq 2^{-n} + \log(2)n2^{-\lambda n} .
$$
Hence $\sum_{n}\eta_{n, p}\leq 2 + \log(2) \dfrac{2^{-\lambda}}{(1- 2^{-\lambda})^2}$. Therefore
\begin{eqnarray}\label{eq: est112}
\sup_{p \in \mathbb{N}^{\star}}\left( \sum_{n}\eta_{n, p}\right)  \leq 2 \delta_{H, K} + 2 + \log(2) \dfrac{2^{-\lambda}}{(1- 2^{-\lambda})^2}=:S_{H, K}.
\end{eqnarray}
Using the fact that $c_p \leq c \sqrt{p}$, for all $p \geq 1$ and by choosing $R >c C_{H, K}$, we get by applying Markov's inequality together with (\ref{eq: estimation moments})
\begin{eqnarray*} 
\mathbb{P}(Y_{n, p} >R \sqrt{p}) \leq  \dfrac{\E[Y_{n, p}^p]}{R^p p^{p/2}} \leq \left( \dfrac{cC_{H, K}}{R} \right)^p \eta_{n, p}.
\end{eqnarray*}
Taking account of (\ref{eq: est112}), we obtain 
\begin{eqnarray} \label{eq: integer}
\sum _{p=1}^{+\infty}\sum _{n=0}^{+\infty} \mathbb{P}(Y_{n, p} >R \sqrt{p})   \leq  \sum _{n=0}^{+\infty} \eta_{n, p} \sum _{p=1}^{+\infty}\left( \dfrac{cC_{H, K}}{R} \right)^p  \leq  \dfrac{R S_{H, K}}{R-cC_{H, K}} <\infty.
\end{eqnarray}
The double parameter Borel-Cantelli lemma  implies  that a.s.
\[ \sup_{p \in \mathbb{N}^{\star}} \dfrac{1}{\sqrt{p}} \sup_{n \in \mathbb{N } }Y_{n, p} <+\infty.
\]
Now for real $p\geq 1$, let  $C:=\sup_{m \in \mathbb{N}^{\star}} \dfrac{1}{\sqrt{m}} \sup_{n \in \mathbb{N } }Y_{n, m}.$ Then $C<\infty$ a.s.  and for all integer $n$, we have 
\[Y_{n, p} \leq \left( Y_{n, [p]+1}\right)^{\frac{p}{[p]+1}}\leq \left(C \sqrt{[p]+1}\right)^{\frac{p}{[p]+1}}.\]
Using the elementary inequality $a^{\alpha} \leq 1 +a$ valid for all $a\geq 0$ and $0< \alpha \leq 1$, we get
\[\dfrac{Y_{n, p} }{\sqrt{p}}\leq \dfrac{1}{\sqrt{p}} + C \sqrt{\dfrac{[p] +1}{p}} \leq 1  + C \sqrt{2}. \]
Therefore, almost surely
\[ \sup_{p \geq 1} \dfrac{1}{\sqrt{p}} \sup_{n \in \mathbb{N } }Y_{n, p} \leq 1  + C \sqrt{2} <+\infty.
\]
This finishes the proof of the second statement in (\ref{reg:bo}) and whence the proof of Theorem \ref{Lei-NualartProcess-result2}.
\noindent \textbf{Proof  of Lemma \ref{lem:1}}:  Let $0<h<1/2$ and assume first that $t\in [0, h]$.  In view of (\ref{eq:moments}) and since $K<2$, we have
\begin{eqnarray}\label{eq: varF}
\var(F_h(t))&\leq & \Gamma(2-K)\int_{0}^{(t+h)^{2h}-t^{2H}} \int_{0}^{(t+h)^{2h}-t^{2H}} (u+v)^{K-2} du dv \nonumber\\
&=&\Gamma(2-K)C^{\prime}_K ((t+h)^{2h}-t^{2H})^K,
\end{eqnarray}
where we have used the change of variables  $u=((t+h)^{2H}-t^{2H})x$ and  $v=((t+h)^{2H}-t^{2H})y$ to obtain the last inequality and the constant $C^{\prime}_K$ is given by
\begin{eqnarray*} 
C^{\prime}_K:= \int_{0}^{1}\int_{0}^{1} (x+y)^{K-2} dx dy = \begin{dcases}
	2 \log(2),&\quad  \mbox{if}\quad \,  K=1; \\
	\dfrac{2^K-2}{K(K-1)},&   \mbox{if}\quad \,   K\neq 1.
	\end{dcases}  
\end{eqnarray*}
Since $(t+h)^{2h}-t^{2H} \leq (2h)^{2H}$, recall that $t\leq h$, we deduce  from (\ref{eq: varF}) that
\[ \var(F_h(t))\leq 2^{2HK} \Gamma(2-K)C^{\prime}_K h^{2HK},\]
and thus by (\ref{eq:moments}) that
 \[
\E(|F_h(t)|^p) \leq \left(c_p 2^{HK} \sqrt{\Gamma(2-K)C^{\prime}_K}\right)^p h^{HKp}=:(c_pA_{H, K})^p  h^{HKp}.\]
Consider now the case  $t\in [h, 1- h]$. Since $x\mapsto x^{K-2}$ is a non-increasing function,  we get
\begin{eqnarray*}
\var(F_h(t))\leq \Gamma(2-K)2^{K-2}t^{2H(K-2)} ((t+h)^{2h}-t^{2H})^2.
\end{eqnarray*}
Using the elementary inequality $(t+h)^{2h}-t^{2H} \leq H2^{1 \vee 2H} t^{2H-1} h=:D_Ht^{2H-1}h$, valid since $h\leq t$, we deduce
\begin{eqnarray*}
\var(F_h(t))\leq  \Gamma(2-K)2^{K-2}t^{2H(K-2)} D_H^2 t^{2(2H-1)}= \Gamma(2-K)2^{K-2} D_H^2t^{-2(1-HK)}h^2,
\end{eqnarray*}
Consequently, in view of  (\ref{eq:moments}), we obtain
 \[
\E(|F_h(t)|^p) \leq \left(c_p 2^{\frac{K-1}{2}} D_H \sqrt{\Gamma(2-K)}\right)^p t^{-p(1-HK)}=:(c_pB_{H, K})^p  h^{p}t^{-p(1-HK)}.\]
This completes the proof of Lemma  \ref{lem:1}.
\vspace{.5cm}

Before formulating the statement of the Besov regularity of the sample paths of the bfBm, we briefly report the original proof in \cite{NB}, but observe the following three differences with respect to our result
\begin{itemize}
\item[(i)] The proof employed  in \cite{NB} relies  on the Ciesielski et al. \cite{CKR}'s isomorphism characterization of the Besov spaces $\mathcal{B}^{HK}_{p, \infty}([0, 1])$ in terms of  sequence spaces. In contrast, our proof has the advantage of being  more direct and show that the Besov-Orlicz regularity is naturally inherited from that of the fBm.
\item[(ii)] The statement made  in \cite{NB} is valid only  for $H\in(0, 1)$ and $K\in (0, 1)$, whereas our result  covers all parameters  $H\in(0, 1)$ and $K\in(0, 2]$ such that $HK\in (0, 1)$. Hence, extending the rang of parameters $H$ and $K$.
\item[(iii)]In addition, our approach can be applied to establish Besov-Orlicz regularity for a broader class of stochastic processes, including bfBm and sfBm.
\end{itemize} 

We now state our second main result as follows.
\begin{theorem}\label{bfbm-result} Let $H\in(0, 1)$, $K\in (0, 2]$ such that $HK\in (0, 1)$. Then, almost surely the paths of the bfBm $B^{H, K}$ belong  to the Besov-Orlicz space $\mathcal{B}^{HK}_{\Phi_2, \infty}(I)$.
\end{theorem}
\textit{Proof:} The proof of the Theorem \ref{bfbm-result} will be essentially based on Proposition \ref{prop1}. The situation when $K=1$ corresponds to the fBm process already considered in \cite{CKR}. Hence we only consider two cases. 
\begin{itemize}
\item[i)] Case $0<K<1$.  In view of (\ref{LNdecompo}), we have 
\begin{eqnarray}\label{LNdecompo1}
\{\; c_2 B^{HK}_t, \;\; t\in I  \; \}\overset{d}{=}\{\; B^{H, K}_t+c_1X_t^{H, K}, \;\; t\in I \; \},
\end{eqnarray}
where the process $X^{H, K}$  is given by (\ref{LNprocess1}). 
Moreover, according to Ciesielski et al. \cite{CKR}, it is well known that
$$\mathbb{P}(B^{HK}_{\cdot} \in \mathcal{B}^{HK}_{\Phi_2, \infty}(I))=1. $$ 
Consequently, by Proposition \ref{prop1}, we deduce that $ \mathbb{P}(B^{H, K} \in \mathcal{B}^{HK}_{\Phi_2, \infty}(I))=1$ since 
 $ \mathbb{P}(X^{H, K} \in \mathcal{B}^{HK}_{\Phi_2, \infty}(I))=1$ by Theorem \ref{Lei-NualartProcess-result2}. 
\item[ii)] Case $1<K<2$. Using (\ref{BE-ext}) instead of (\ref{LNdecompo}) and arguing as before, the conclusion follows immediately from Proposition \ref{prop1} and Theorem \ref{Lei-NualartProcess-result2}.
\end{itemize}
The proof of Theorem \ref{bfbm-result} is finished.
\vspace{0.2cm}

The following result was stated without proof in \cite{NB}, where the authors claimed that following the same line of calculations using  Ciesielski et al's \cite{CKR} isomorphism characterization as in the bfBm case, analogous conclusions could be obtained  for the sfBm. Here, we give a more general statement and  offer  a direct and  simple proof relying on the well-known Besov-Orlicz regularity result for the trajectories of the fBm together with Proposition \ref{prop1}.

\begin{theorem} \label{sbfBm1}Let $H\in(0, 1)$. Then the paths of the sfBm $S^{H}$ belong  almost surely to the Besov-Orlicz space $\mathcal{B}^{HK}_{\Phi_2, \infty}(I)$.
\end{theorem} 
\textit{Proof:} The argument follows the same line as in the bifractional Brownian motion case and relies on the additive decomposition in law (\ref{subdecomp1}), (\ref{RT-ext}) together with Proposition \ref{prop1}. We distinguish two cases.
\begin{itemize}
\item[i)]Case $0<H<1/2$. From (\ref{subdecomp1}), we have the following decomposition in law
\begin{eqnarray*}
\{S^H_t, \;\; t\geq 0 \; \}&\overset{d}{=}&\{\; B^{H}_t+c_3X_t^{2H}, \quad t\geq 0 \; \}\\
&=&\{\; B^{H}_t+c_3X_t^{2H, 1/2}, \quad t\geq 0 \; \}.
\end{eqnarray*}
As before $\mathbb{P}(B^{H}_{\cdot} \in \mathcal{B}^{H}_{\Phi_2, \infty}(I))=1. $ 
Consequently, by Proposition \ref{prop1}, we deduce that $$ \mathbb{P}(S^{H} \in \mathcal{B}^{H}_{\Phi_2, \infty}(I))=1,$$
 since 
 $ \mathbb{P}(X^{2H, 1/2} \in \mathcal{B}^{H}_{\Phi_2, \infty}(I))=1$ by Theorem \ref{Lei-NualartProcess-result2}. 
 \item[ii)]Case $1/2<H<1$. From (\ref{RT-ext}), we have the following decomposition in law
\begin{eqnarray*}
\{B^{H}_t, \;\; t\geq 0 \; \}&\overset{d}{=}&\{\; S^H_t+c_4X_t^{2H}, t\geq 0 \; \},\\
&=&\{\; S^{H}_t+c_4X_t^{2H, 1/2}, \quad t\geq 0 \; \}.
\end{eqnarray*}
The result follows again in this case by using  Proposition \ref{prop1}, and the fact that  $ \mathbb{P}(X^{2H, 1/2} \in \mathcal{B}^{H}_{\Phi_2, \infty}(I))=1$. 
\end{itemize}
This finishes the proof of Theorem \ref{sbfBm1}.\\

Following the same line of thoughts, we now state an analogue  result for the processes $G$.
%This type of argument may be applied to other situation. Here we %present 
\begin{theorem}\label{G} Let  $1/2<H<1$ and $\gamma\in(0, 2H)$. Let $G$ be the centered Gaussian process with the covariance function given by (\ref{}). Then, almost all the paths of  $G$ belong to the Besov-Orlicz space $\mathcal{B}^{H-\gamma/2}_{\Phi_2, \infty}(I)$. 
\end{theorem} 
\textit{Proof:}  Notice first that by (\ref{HNdecompo}) the law of $G$ can be split into the sum  
\begin{eqnarray}\label{hn-decomp}
\{\; G_t, t\geq 0 \; \}\overset{d}{=}\{\; \sqrt{\kappa} B^{H- \gamma/2}_t+\sqrt{\lambda} X_t^{4H-2 \gamma+1, 1/2}, \quad t\geq 0 \; \}.
\end{eqnarray}
From Theorem \ref{Lei-NualartProcess-result2} we obtain $$\mathbb{P}( X_{\cdot}^{4H-2 \gamma+1, 1/2} \in \mathcal{B}^{2H- \gamma+1/2}_{\Phi_2, \infty}(I))=1.$$
Since $2H -\gamma +1/2 > H-\gamma/2$, the conclusion follows from (\ref{hn-decomp}), the embedding (\ref{BO-embeding}) and the classical result for the fractional Brownian motion.

\end{document}